\documentclass[a4paper,12pt]{amsart}

\usepackage{epsfig}
\usepackage{mathrsfs}
\usepackage{amsthm,amsfonts}
\usepackage{amssymb,graphicx,color}
\usepackage[all]{xy}

\newtheorem{theorem}{Theorem}[section]
\newtheorem{lemma}[theorem]{Lemma}
\newtheorem{corollary}[theorem]{Corollary}
\newtheorem{proposition}[theorem]{Proposition}
\newtheorem{definition}[theorem]{Definition}
\newtheorem{example}[theorem]{Example}
\newtheorem{remark}[theorem]{Remark}

\newcommand{\C}{\mathbb C}
\newcommand{\R}{\mathbb R}

\newcommand{\Z}{\mathbb Z}

\makeatletter
\newcommand{\tpitchfork}{%
  \vbox{
    \baselineskip\z@skip
    \lineskip-.52ex
    \lineskiplimit\maxdimen
    \m@th
    \ialign{##\crcr\hidewidth\smash{$-$}\hidewidth\crcr$\pitchfork$\crcr}
  }%
}
\newcommand{\dist}{\operatorname{dist}}
\newcommand{\Grad}{\operatorname{Grad}}
\makeatother

\begin{document}

\title[ Topological classification and finite determinacy of knotted maps.
 ]{Topological classification and finite determinacy of knotted maps}

\author{ Juan J. Nu\~no-Ballesteros}

\address{Departament de Matem\`atiques,
Universitat de Val\`encia, Campus de Burjassot, 46100 Burjassot,
Spain}

\email{Juan.Nuno@uv.es}

\author{ Rodrigo Mendes}

\address{Instituto de Ci\^encias exatas e da natureza, Universidade de Integra\c{c}\~ao Internacional da Lusofonia Afro-Brasileira (UNILAB), Campus dos Palmares, Cep. 62785-000. Acarape-Ce,
Brasil}

 \email{rodrigomendes@unilab.edu.br}

\thanks{The first author has been partially supported by DGICYT Grant MTM2015--64013--P}

\begin{abstract}
 We show that the knot type of the link of a real analytic map germ with isolated singularity $f\colon(\mathbb{R}^2,0)\to(\mathbb{R}^4,0)$ is a complete invariant for $C^0$-$\mathscr A$-equivalence. Moreover, we also prove that isolated instability implies $C^0$-finite determinacy, giving an explicit estimate for its degree. For the general case of real analytic map germs, $f\colon (\mathbb{R}^n,0)  \rightarrow (\mathbb{R}^p,0)$  ($n \leq p$), we use the Lojasiewicz exponent associated to the Mond's  double point ideal $I^2(f)$ to obtain some criteria of Lipschitz and analytic regularity.
\end{abstract}

\maketitle

\section{Introduction}

In a previous paper \cite{MB}, we consider analytic map germs $f\colon(\mathbb{R}^2,0)\rightarrow (\mathbb{R}^4,0)$ with isolated singularity. This means that there exists a representative $f:\mathcal{U}\rightarrow \mathcal{V}$ such that $f$ is a topological embedding on $\mathcal{U}$ and an immersion on $\mathcal{U}\setminus \{0\}$. In particular, its image $X = f(\mathcal{U})$ is a surface with isolated singularity surface at the origin in $\mathbb{R}^4$. By the cone structure theorem, the topological type of the germ $f$ is determined by the knot type of its link $K(f) = X \cap \mathbb{S}^3_\epsilon$ (where $\epsilon>0$ is a Milnor-Fukuda radius of $f$). In the first part of this paper, we show the converse of this, namely, that the knot type of the link $K(f)$ is a topological invariant. 

In the second part of the paper, we consider the general case of analytic map germs $f\colon(\mathbb{R}^n,0)\rightarrow (\mathbb{R}^p,0)$ with isolated singularity ($n\le p$). We define a pair of invariants which are Lojasiewicz exponents associated to the Mond ideal $I^2(f)$. Recall that there exists a $p\times n$-matrix $\alpha$ with entries in $\mathscr E_{2n}$ such that
\[
f(x')-f(x)=\alpha(x,x')(x'-x),
\]
and that the ideal $I^2(f)\subset \mathscr E_{2n}$ generated by the functions $f_i(y)-f_i(x)$, with $i=1,\dots,p$, and by the $n\times n$-minors of $\alpha$ (here $\mathscr E_{2n}$ is the local ring of analytic function germs from $(\R^n\times\R^n,0)$ to $\R$). 

The first invariant $\mathcal{L}_0(\tilde{\Delta}f)$ is called the isolated singularity exponent and is defined as the Lojasiewicz exponent of $I^2(f)$ with respect to the maximal ideal $\mathcal M_{2n}$. It has the property that $\mathcal{L}_0(\tilde{\Delta}f)<\infty$ if and only if $f$ has isolated singularity. We use this invariant to prove that if $f\colon(\mathbb{R}^2,0)\rightarrow (\mathbb{R}^4,0)$ has isolated singularity, then $f$ is $C^0$-finitely determined. In fact, we give an explicit estimate for the degree of $C^0$-determinacy. It is well known that any finitely determined germ has isolated instability (and in particular, isolated singularity when $n=2$ and $p=4$), by the Mather-Gaffney criterion (see \cite{Wall}). But the converse is not true in the real analytic case. A natural open question is if isolated instability implies $C^0$-finite determinacy. This is known to be true for function germs (see \cite{Kuo}) and here we answer this question in the case $n=2$ and $p=4$. For the general case of $C^\infty$ map germs $f\colon(\mathbb{R}^n,0)\rightarrow (\mathbb{R}^p,0)$, we refer to \cite{BIW}, where they give a sufficient condition for $C^0$-finite determinacy in terms of some Lojasiewicz inequalities in the jet space.

The second invariant $\mathcal{L}_0({\Delta}f)$ is called the double point exponent and is defined as the Lojasiewicz exponent of the ideal generated by $f_i(y)-f_i(x)$, $i=1,\dots,p$, with respect to the ideal generated by $x'_j-x_j$, $j=1,\dots,n$. In this case, we have that $\mathcal{L}_0({\Delta}f)<\infty$ if and only if $f$ is injective. We show that this a bi-Lipschitz invariant of $f$ and we use it to prove that $f$ is a bi-Lipschitz embedding if and only if it is a smooth embedding. This result could be seen as a weaker real version of a theorem by Birbrair, L\^e, Fernandes and Sampaio \cite{BFLS}, where they show that if a complex algebraic set $X\subset\C^n$ is Lipschitz regular at a point $x_0\in X$, then $X$ is smooth at $x_0$.

%
%
%

\section{Invertible cobordisms of knots}

In this section, we show that two knots which are invertible cobordant from both ends are equivalent. We
first recall the notion of peripheral structure of a knot. Along all the paper, a knot $K \subset \mathbb{S}^3=\{x \in \mathbb{R}^4; \|x\|=1\}$ is always a tame knot, unless otherwise stated. Two knots $K_1, K_2$ are said equivalent if there exist a homeomorphism $\phi:\mathbb{S}^3 \rightarrow \mathbb{S}^3$ such that $\phi(K_1)=K_2$.

\begin{definition} \emph{Let $K \subset \mathbb{S}^3$ be a knot. Consider $N(K) \subset \mathbb{S}^3$ a tubular neighbourhood of $K$, so $N(K)$ is a smooth submanifold homeomorphic to $K \times\mathbb{B}_2$. A \emph{meridian} of $K$ is a simple closed curve $m$ contained in $\partial N(K)$ such that $m$ is not homotopically trivial in $\partial N(K)$, but is homologically trivial in $N(K)$. Analogously, a \emph{longitude} of $K$ is a simple closed curve $l$ contained in $\partial N(K)$ such that $l$ is homologous to $K$ in $N(K)$. We say that $(m,l)$ is a \emph{meridian-longitude pair} of $K$.
}
\end{definition}

\begin{definition}
\emph{Let $K_1, K_2 \subset \mathbb{S}^3$ be two knots and consider a meridian-longitude pair $(m_i,l_i)$ of each $K_i$ contained in  $\partial N(K_i)$, $i=1,2$. Let $\langle[m_i],[l_i]\rangle$ be the subgroup generated by their  classes in $G_i = \pi_1(\mathbb{S}^3\setminus \mathring{N}(K_i))$. Suppose that there is an isomorphism  $\varphi_*:G_1 \rightarrow G_2$. We say that the isomorphism $\varphi_*$ \emph{preserves the peripheral structure} when $\varphi_*(\langle[m_1],[l_1]\rangle)$ is conjugate to a subgroup of $\langle[m_2],[l_2]\rangle$ in $G_2$.}
\end{definition}

The following result is a consequence of the works of Dehn \cite{Dehn}, Waldhausen \cite{wald} and Gordon and Luecke \cite{GL}. It means that the knot group plus the peripheral structure information provide a complete invariant of the knot.

\begin{theorem}\label{peripheralstructure}
Let $K_1, \ K_2 \subset \mathbb{S}^3$ be two knots such that their knot groups $G_1$ and $G_2$ are isomorphic. If the isomorphism $\varphi_*:G_1 \rightarrow G_2$ preserves the peripheral structure then $K_1$ and $K_2$ are equivalent.
\end{theorem}

\begin{proof}
If one of the knots is trivial, then $G_1\cong G_2 \cong\Z$ and by Dehn's Lemma \cite{Dehn}, the other knot is also trivial. Thus, we can assume that $K_1$ and $K_2$ are not trivial. Again by Dehn's Lemma, we obtain that the 3-manifolds $\mathbb{S}^3\setminus \mathring{N}(K_i)$ are sufficiently large, for $i = 1,2$. By Waldhausen theorem (see \cite[Corollary 6.5]{wald}), there exists a homeomorphism between $\mathbb{S}^3\setminus \mathring{N}(K_1)$ and $\mathbb{S}^3\setminus \mathring{N}(K_2)$, which can be extended to a homeomorphism between $\mathbb{S}^3\setminus K_1$ and $\mathbb{S}^3\setminus K_2$. So, the knots $K_1$ and $K_2$ are equivalent, by Gordon-Luecke theorem \cite{GL}.
\end{proof}

In a previous paper \cite{MB}, the authors used the notion cobordism of knots associated to two $C^0$-$\mathscr A$-equivalent map germs, but here we are interested in an equivalence which is a little bit stronger: the notion of invertible cobordism from both ends. This appears in a more general context in the works of Stallings  \cite{Sta} and Siebenmann \cite{Sie}. We recall the definition in the case of knots.

\begin{definition} 
\emph{Two knots $K_1$, $K_2 \subset \mathbb{S}^3$ are called \emph{invertible cobordant} from end $K_2$ if there exist cobordisms $(W_{12};K_1,K_2)$ and $(W_{21};K_2,K_1)$ such that $(W_{12}\cup W_{21};K_1,K_1)$ is homeomorphic to the product cobordism $(K_1 \times  [0,1],K_1,K_1)$. In the same way, we can define invertible cobordism from end $K_1$. }
\end{definition}

\begin{proposition}\label{invertible}
If two knots $K_1$, $K_2 \subset \mathbb{S}^3$ are invertible cobordant from both ends, then $K_1$, $K_2$ are equivalent.
\end{proposition}

\begin{proof}
Let $(W;K_1,K_2)$ be an invertible cobordism from both ends $K_1$ and $K_2$. By \cite[Proposition 2.1]{summers}, $\mathbb{S}^3 \times  [0,1]{\setminus}W$ provides an $h$-cobordism between $\mathbb{S}^3{\setminus}K_1$ and $\mathbb{S}^3{\setminus}K_2$. It means that the inclusions
\[
i_1:\mathbb{S}^3{\setminus}K_1 \to \mathbb{S}^3 \times  [0,1]{\setminus}W, \quad i_2:\mathbb{S}^3{\setminus}K_2 \to  \mathbb{S}^3 \times  [0,1]{\setminus}W
\]
are homotopy equivalences. Consider the retraction $r_2:\mathbb{S}^3 \times  [0,1]{\setminus}W \to  \mathbb{S}^3{\setminus}K_2$. Notice that $r_2 \circ i_1$ induces an isomorphism between the fundamental groups $\pi_1(\mathbb{S}^3{\setminus}K_1)$ and $\pi_1(\mathbb{S}^3{\setminus}K_2)$. If $K_1$ is a trivial knot, then $\pi_1(\mathbb{S}^3{\setminus}K_1)=\mathbb{Z}=\pi_1(\mathbb{S}^3{\setminus}K_2)$ and thus, $K_2$ is also a trivial knot. Hence, we may assume that $K_1$ is not a trivial knot. In this case, by Theorem \ref{peripheralstructure}, it is enough to show that the isomorphism $(r_2 \circ i_1)_*$ preserves the peripheral structure.

Let $m_1$ be a meridian of the knot $K_1$. We may suppose that $r_2 \circ i_1(m_1) \subset \partial N(K_2)$ for some solid torus $N(K_2)$ containing the knot $K_2$. By the isomorphism condition, $[r_2 \circ i_1(m_1)]$ is not zero in $\pi_1(\mathbb{S}^3{\setminus}K_2)$ because $[m_1]\neq 0$ in $\pi_1(\mathbb{S}^3{\setminus}K_1)$. Hence, it suffices to show that $r_2 \circ i_1(m_1)$  is homologically trivial in $N(K_2)$. Otherwise, we have that $[r_2 \circ i_1(m_1)]$ would be a multiple of $[K_2]$ in $H_1(N(K_2))\cong\mathbb{Z}$. Since $[K_2] = 0$ in $H_1(\mathbb{S}^3{\setminus}K_2)$, this would imply that $0 = [r_2 \circ i_1(m_1)] \in H_1(\mathbb{S}^3\setminus K_2))$, which is a contradiction (because the homology class is a homotopy invariant). Hence, $r_2\circ i_1(m_1)$ is a meridian of $K_2$.

Now, let $l_1$ be a longitude in $\partial N(K_1) \subset \mathbb{S}^3\setminus \mathring{N}(K_1)$. Consider $l_2=r_2(i_1(l_1))$. We may assume that $l_2 \subset \partial N(K_2)$ for some solid torus $N(K_2)$. Since $K_1$ is  not trivial, by Dehn's lemma, the inclusion $i:\partial N(K_1) \rightarrow \mathbb{S}^3{\setminus}\mathring{N}(K_1)$ induces a monomorphism ${i}_{*}:\pi_1(\partial N(K_1))=\mathbb{Z}\oplus \mathbb{Z} \rightarrow \pi_1(\mathbb{S}^3\setminus \mathring{N}(K_1))$, where the pair meridian-longitude $(m_1,l_1)$ of $K_1$ provide generators of $\mathbb{Z}\oplus \mathbb{Z}$. In particular, $l_1$ is homologous to the knot $K_1$ in $N(K_1)$. Now, we have that $\langle {r_2}_{*}([m_1],[l_1])\rangle=\mathbb{Z}\oplus \mathbb{Z}  \le \pi_1(\mathbb{S}^3\setminus \mathring{N}(K_2))$. Since ${r_2}_{*}(m_1)$ is a meridian, the other class ${r_2}_{*}(l_1)=[r_2(l_1)]$ is, necessarily, homologous to the knot $K_2$.
\end{proof}

\section{$C^0$-$\mathscr A$-equivalence and knot type}

In a previous paper \cite{MB}, we consider the $C^0$-$\mathscr A$-classification of analytic map germs $f\colon(\mathbb{R}^2,0)\rightarrow (\mathbb{R}^4,0)$ with isolated singularity. One of the main results of \cite{MB} is that two map germs are $C^0$-$\mathscr A$-equivalent if their links have the same knot type. In this section, we will prove that the converse is also true, that is, if two map germs are $C^0$-$\mathscr A$-equivalent, then their links are equivalent as knots.

\begin{definition}\label{iso-sing}
\emph{An analytic map germ $f\colon(\mathbb{R}^n,0)\rightarrow (\mathbb{R}^p,0)$ has \emph{isolated singularity at $0$} (or isolated instability at 0) when for some representative $f\colon\mathcal{U} \subset \mathbb{R}^n \rightarrow \mathcal{V} \subset \mathbb{R}^p$, $f$ is a immersion on $\mathcal{U}-\{0\}$ and a $C^0$-embedding on $\mathcal{U}$. In particular, its image $f(\mathcal{U})$ is a topological submanifold with isolated singularity at $0$.}
\end{definition}

Recall that two $C^\infty$ map germs $f,g\colon(\R^n,0)\to(\R^p,0)$ are said $\mathscr A$-equivalent if there exist diffeomorphisms $\phi,\psi$ such that $g=\psi\circ f\circ \phi^{-1}$. If $\phi,\psi$ are homeomorphisms instead of diffeomorphisms, then we say that $f,g$ are $C^0$-$\mathscr A$-equivalent.

Given $f\colon(\R^n,0)\to(\R^p,0)$ and a representative  $f\colon\mathcal U\to \mathcal V$  we denote:
\begin{align*}
&D^p_\epsilon=\{y\in\R^p: \|y\|^2\le\epsilon\},\quad S^{p-1}_\epsilon=\{y\in\R^p: \|y\|^2=\epsilon\},\\
&\tilde D^n_\epsilon=f^{-1}(D^p_\epsilon),\quad \tilde S^{n-1}_\epsilon=f^{-1}(S^{p-1}_\epsilon).
\end{align*}

The cone structure theorem is true for analytic map germs $f\colon(\mathbb{R}^n,0)\rightarrow (\mathbb{R}^p,0)$ with isolated instability. Here we state an adapted version for the case $n=2$ and $p=4$.

\begin{theorem}[Cone structure theorem]\label{conestruct} \cite{MB,fukuda}
Let $f\colon(\R^2,0)\to(\R^4,0)$ be an analytic map germ with isolated singularity. There exists $\epsilon_0>0$ such that for any $\epsilon$, with $0<\epsilon\le \epsilon_0$ we have:
\begin{enumerate}
\item $\tilde S^1_\epsilon$ is diffeomorphic to the circle $\mathbb S^1$.
\item $f|_{\tilde S^1_\epsilon}:\tilde S^1_\epsilon\to S^3_\epsilon$ is an embedding, whose $\mathscr A$-class is independent of $\epsilon$.
\item $f|_{\tilde D^2_\epsilon\setminus\{0\}}:\tilde D^2_\epsilon\setminus\{0\}\to D^4_\epsilon\setminus\{0\}$ is $\mathscr A$-equivalent to the product map $\text{\rm id}\times f|_{\tilde S^1_\epsilon}: (0,\epsilon]\times \tilde S^1_\epsilon\to  (0,\epsilon]\times S^3_\epsilon$.
\item $f|_{\tilde D^2_\epsilon}:\tilde D^2_\epsilon\to D^4_\epsilon$ is $C^0$-$\mathscr A$-equivalent to the cone of $f|_{\tilde S^1_\epsilon}:\tilde S^1_\epsilon\to S^3_\epsilon$.
\end{enumerate}
\end{theorem}
We say that the number $\epsilon_0>0$, in Theorem \ref{conestruct}, is a \emph{Milnor-Fukuda radius} for $f$. Essentially, the properties (1), (2), (3)  and (4) are obtained from the condition of transversality between $f$ and $S^3_\epsilon$, for all $\epsilon$ such that $0<\epsilon\leq \epsilon_0$. Notice that all knots $K_\epsilon(f)=f(\tilde S^1_\epsilon) \subset S^3_\epsilon$ are equivalent, for $0<\epsilon<\epsilon_0$. The class of the knot $K_\epsilon(f)$ is denoted by $K(f)$.

\begin{theorem}\label{equivalence}
Let $f,g\colon(\mathbb{R}^2,0)\to(\mathbb{R}^4,0)$ be analytic map germs with isolated singularity. Then the following conditions are equivalent:
\begin{enumerate}
\item $f$ is $C^0$-$\mathscr A$-equivalent to $g$;
\item $K(f)$ and $K(g)$ are equivalent as knots.
\end{enumerate}
\end{theorem}
\begin{proof} $(2)\Rightarrow (1)$: Let $\epsilon<\min\{\epsilon_1,\epsilon_2\}$, where $\epsilon_1$ is a Milnor-Fukuda radius for $f$ and $\epsilon_2$ is a Milnor-Fukuda radius for $g$. Since $K(f)$ and $K(g)$ are equivalent, there exists a homemomorphism $\phi_\epsilon: S^3_\epsilon \rightarrow S^3_\epsilon$ such that $\phi_\epsilon(K_\epsilon(f))=K_\epsilon(g)$. Let $\tilde S^1_\epsilon=f^{-1}(S^{3}_\epsilon)$ and $\tilde {\tilde {S}}^1_\epsilon=g^{-1}(S^{3}_\epsilon)$ and consider the map $\tilde \phi:\tilde S^1_\epsilon \rightarrow \tilde {\tilde {S}}^1_\epsilon $ given by $\tilde \phi_\epsilon=g^{-1} \circ \phi_\epsilon \circ f|_{\tilde S^1_\epsilon}$. Thus, we obtain $g|_{\tilde {\tilde {S}}^1_\epsilon}= \phi_\epsilon \circ f \circ \tilde \phi_\epsilon^{-1}$. Hence, $f|_{\tilde S^1_\epsilon}$ and $g|_{{\tilde {\tilde S}}^1_\epsilon}$ are $C^0{-}\mathscr A$-equivalent and, by condition (4) in Theorem \ref{conestruct}, it follows that $f$ is $C^0$-$\mathscr A$-equivalent to $g$.

\medskip
Proof that $(1) \Rightarrow (2)$:
Let $\psi\colon(\mathbb{R}^4,0) {\rightarrow }(\mathbb{R}^4,0)$ and $\phi\colon(\mathbb{R}^2,0)\to  (\mathbb{R}^2,0)$ be homeomorphisms such that $g=\psi \circ f \circ \phi^{-1}$. By \cite[proposition 2.8]{MB}, we have that $K(g)$ and $\psi(K(f))$ are cobordant. We will prove that $K(g)$ and $\psi(K(f))$ are in fact invertible cobordant from both ends and thus, $K(f)$ and $K(g)$ are equivalent by Proposition \ref{invertible}.

Let $K(f)=K_{\epsilon_0}(f)$, where $\epsilon_0$ is a Milnor-Fukuda radius of $f$. Let ${D}^4_\delta \subset \psi(\mathring D_{\epsilon_0}^4)$, where $\delta$ is a Milnor-Fukuda radius of $g$ and let $\epsilon_1<\epsilon_0$ such that $D^4_{\epsilon_1} \subset \psi^{-1}(\mathring{D}^4_\delta)$. Then,
\[
\phi(\tilde D^2_{\epsilon_1}) = \phi(f^{-1}(D^4_{\epsilon_1}))\subset\phi(f^{-1}(\psi^{-1}(\mathring D^4_\delta)))=g^{-1}(\mathring D^4_\delta)=\mathring{ \tilde D}^2_{\delta}.
\]
We know that $\phi(\tilde D^2_{\epsilon_0})-\mathring{\tilde D}_\delta^2$ is homeomorphic to $\mathbb{S}^1\times [0,1]$, $\psi(D_{\epsilon_0}^4) - \mathring{D}_\delta^4$ is homeomorphic to $\mathbb{S}^3\times [0,1]$ and that $g$ is an embedding from $\phi(\tilde D^2_{\epsilon_0})-\mathring{\tilde D}_\delta^2$ to $\psi(D_{\epsilon_0}^4) - \mathring{D}_\delta^4$. So, $W_{12}=g(\phi(\tilde D^2_{\epsilon_0})-\mathring{\tilde D}_\delta^2)$ defines a cobordism between $\psi(K(f))$ and $K(g)=K_\delta(g)$. Analogously,
$W_{21}=g(\tilde D_\delta^2-\phi(\mathring{\tilde D}^2_{\epsilon_1}))$ gives a cobordism between $K(g)$ and $\psi(K(f))$. The union $W_{12}\cup W_{21}$ is given by $g(\phi(\tilde D^2_{\epsilon_0} -\mathring{\tilde D}^2_{\epsilon_1}))=\psi(f(\tilde D^2_{\epsilon_0} -\mathring{\tilde D}^2_{\epsilon_1}))$, which is trivial by condition (3) of Theorem \ref{conestruct}. Thus, we have an invertible cobordism from end $\psi(K(f))$. The invertible cobordism from end $K(g)$ is obtained in a similar way.
\end{proof}

\section{The isolated singularity exponent}

Given an analytic map germ $f\colon(\mathbb{R}^n,0)\rightarrow (\mathbb{R}^p,0)$, we denote by $j^kf(0)$ its $k$-jet, that is, its Taylor expansion of order $k$. We recall that $f$ is called $k$-determined (resp. $C^0$-$k$-determined) if for any other map germ $g$ such that $j^kf(0)=j^kg(0)$, $g$ is $\mathscr A$-equivalent to $f$ (resp. $C^0$-$\mathscr A$-equivalent to $f$). One says that $f$ is finitely determined (resp. $C^0$-finitely determined) if it is $k$-determined (resp. $C^0$-$k$-determined) for some $k$.

The isolated singularity condition (in the sense of Definition \ref{iso-sing}) is a generic condition, not only for map germs $f\colon(\mathbb{R}^2,0) \rightarrow (\mathbb{R}^{4},0)$, but more generally for map germs $f\colon(\mathbb{R}^n,0) \rightarrow (\mathbb{R}^p,0)$, with $p\ge 2n$. 
In fact, if such a map germ is finitely determined, then, by the Mather-Gaffney finite determinacy criterion (see \cite{Wall}), $f$ has isolated instability. But when $p\ge 2n$, this means that $f$ is a $C^0$-embedding in $\mathcal{U}$ and a immersion on $\mathcal{U}-\{0\} \subset \mathbb{R}^n$, for some representative $f\colon \mathcal U\to\mathcal V$. In this case, the image of $f$ is a $n$-topological manifold with isolated singularity embedded in $\mathbb{R}^{p}$. In this section, we will use the Lojasiewicz exponent as an invariant which detects the isolated singularity condition of  a given analytic map germ $f\colon(\mathbb{R}^n,0) \rightarrow (\mathbb{R}^{p},0)$. In other words, the existence of this number guarantees that $M=f(\mathcal{U})\setminus \{0\}$ is an embedded smooth submanifold.
The Lojasiewicz exponent and inequalities are powerful tools  to investigate the topology and geometry of analytic maps and analytic sets. For instance, see the works \cite{BA, BR,JG,JEL}.

\medskip

Let $f\colon(\mathbb{R}^n,0) \rightarrow (\mathbb{R}^{p},0)$ be an analytic map germ ($n \leq p$). For each $i=1,\dots,p$ there exist analytic functions $\alpha_{ij}\colon (\mathbb{R}^{2n},0)\to\mathbb R$, with $j=1,\dots,n$ such that
\[
f_i(x')-f_i(x)=\sum_{j=1}^n \alpha_{ij}(x,x')(x_j'-x_j).
\]
Let $\alpha=(\alpha_{ij})$ be the $p\times n$ matrix obtained in this way and let
$D_1, D_2, \ldots, D_r$ be the minors of order $n$ of $\alpha$. We define the map germ $\tilde{\Delta}f\colon(\mathbb{R}^{2n},0)\rightarrow \mathbb{R}^{p+r}$ as
\[
\tilde{\Delta} f(x',x)=(f(x')-f(x),D_1(x,x'),\ldots,D_r(x,x')),
\]
where $r=$ ${p}\choose{n}$.

Let $\tilde{\alpha}_0$ be the Lojasiewicz exponent of the map $\tilde \Delta f$. That is, $\tilde{\alpha}_0$ is the infimum of the numbers $\alpha>0$ which satisty the Lojasiewicz inequality
\begin{equation}\label{loja3}
\|\tilde{\Delta} f(x',x)\| \geq C \dist((x',x),V(\tilde{\Delta} f))^{\alpha},
\end{equation}
for some real positive numbers $C,\epsilon$ and for all
$\|(x',x)\| < \epsilon$, where $V(\tilde{\Delta} f)$ is the zero locus of $\tilde{\Delta} f$.

\begin{remark}\label{MatrixMond}
\emph{
The matrix $\alpha$ is not unique, but it is easy to see that $V(\tilde \Delta f)$ does not depend of the choice of $\alpha$. In fact, if $x \neq x'$, then
$\tilde{\Delta} f(x',x)=0$ if and only if $f(x)=f(x')$. Otherwise, if $x=x'$, we have that $\alpha(x,x)=Df(x)$ (the Jacobian matrix of $f$ at $x$), hence
$\tilde{\Delta} f(x',x)=0$ if and only if $x$ is a non-immersive point of $f$. If $f$ is a $C^0$-embedding and an immersion outside of $0$, then $V(\tilde \Delta f)=\{0\}$. Moreover, the inequality (\ref{loja3}) has the following form:}
\begin{equation}\label{singnumber}
\|\tilde{\Delta} f(x',x)\| \geq C \|(x',x)\|^{\alpha}.
\end{equation}
\end{remark}

\begin{definition}
{\rm The number $\tilde \alpha_0$ is defined as the infimum of those $\alpha>0$ which satisfy the inequality (\ref{singnumber}) for some real positive numbers $C,\epsilon$ and for all  $\|(x',x)\| < \epsilon$. The number $\tilde \alpha_0$ is called the \emph{isolated singularity exponent} of $f$. We denote it by $\mathcal{L}_0(\tilde{\Delta}f)$.}
\end{definition}

The isolated singularity exponent $\alpha_0$ does not depend on the choice of the matrix $\alpha$. In fact, by \cite[Proposition 3.1]{M}, the ideal $I^2(f)$ generated in $\mathscr E_{2n}$ by the components of $\tilde \Delta f$ is independent of the choice of $\alpha$. Moreover, $\alpha_0$ is the Lojasiewicz exponent of the ideal $I^2(f)$ with respecto to the maximal ideal in $\mathscr E_{2n}$ and it is well known that the Lojasiewicz exponent is independent of the choice of the system of generators of the ideals.

\begin{remark}
\emph{
 By using the sum norm, we may write the inequality (\ref{singnumber}) as follows:
\[
 \frac{\|\tilde{\Delta} f(x',x)\|}{\|(x',x)\|^{\alpha}}=\frac{\sum_{j=1}^p|f_j(x')-f_j(x)|}{\|(x',x)\|^{\alpha}}+\frac{\sum_{i=1}^r|D_i(x,x')|}{\|(x',x)\|^{\alpha}} \geq C,
 \]
for some $C,\epsilon>0$ and for all $0<\|(x',x)\| < \epsilon$. Along the diagonal we have
\[
\frac{\|\tilde{\Delta} f(x,x)\|}{\|(x,x)\|^{\alpha}}=\frac{\sum_{i=1}^r|D_i(x,x)|}{\|(x,x)\|^{\alpha}}.
\]
So we need to control only of the term $\sum_{i=1}^r|D_i(x,x)|/\|(x,x)\|^{\alpha}$ in order to obtain $\mathcal{L}_0(\tilde{\Delta}f)$.}
 \end{remark}

\begin{proposition}\label{is}
An analytic map germ $f$ has isolated singularity at $0$ if and only if $\mathcal{L}_0(\tilde{\Delta}f)<\infty$.
\end{proposition}
\begin{proof}
If $f$ has isolated singularity, then $\mathcal{L}_0(\tilde{\Delta}f)<\infty$ by construction. Suppose now that $\mathcal{L}_0(\tilde{\Delta}f)=\tilde\alpha_0<\infty$. Let $f(x)=f(x')$. Then
\[
0=\|f(x')-f(x),D_1(x,x'),\ldots,D_r(x,x')\| \geq C\|(x,x')\|^{\alpha}\geq 0,
\]
so $x=x'=0$. Hence, $f$ is a $C^0$-embedding for some representative. Suppose now that $x$ is a non-immersive point of $f$. Taking $x=x'$ in the previous inequality, we have that $0=\|D_1(x,x),\ldots,D_r(x,x)\| \geq \|(x,x)\|^{\alpha}\geq 0$, Hence, $x=0$. Thus, $f$ has isolated singularity at $0$.
\end{proof}

  In \cite{MB}, we introduced the invariant
\[
\delta(f):=\dim_\mathbb{R}\frac{\mathscr E_4}{I^2(f)},
\]
where $\mathscr E_n$ is the local ring of analytic function germs $(\R^n,0)\to\R$. This number detects the finite determinacy property for the map germ $f$, that is, $f$ is finitely determined if and only if $\delta(f)<\infty$.  Since isolated singularity does not imply finite determinacy, it may happen that $\delta(f)=\infty$, but  $\mathcal{L}_0(\tilde{\Delta}f)<\infty$. For instance, consider the map germ $f\colon(\R^2,0)\to(\R^4,0)$ is given by $f(x,y)=(x,y^2,y(x^2+y^2),0)$.

\begin{remark}
{\rm If $f$ has corank $1$ at $0$, the isolated singularity exponent of $f$ have a more simple description. After analytic change of coordinates, we may write $f$ in the following form:
\[
f(z,y)=(z,\tilde f(z,y)),
\]
where $z \in \mathbb{R}^{n-1}$, $y \in \mathbb{R}$ and $\tilde f(z,y)=(f_n(z,y),\ldots,f_p(z,y))$. Then, the matrix $\alpha(x,x')$ satisfies $\alpha_{ij}(x,x')=1$ for $i=j\leq n$ and $\alpha_{ij}(x,x')=0$ for $i \neq j$, $i,j \leq n$. So, we have that $n-p+1$ minors of $\alpha(x,x')$ are given by $D_j(x,x')=\alpha_{nj}(x,x')=\alpha_{nj}(z,y,u,v), \ j \in \{n\ldots,p\}$. Notice that each $\alpha_{nj}(x,x')$ can be taken as a cofactor of the others minors. Thus, the ideal $I^2(f)$ is generated by $\alpha_{nj}(x,x')$, $j \in \{n,\ldots,p\}$. Now, if we take $z=u$ in $f(z,y)-f(u,v)$ we have
\[
\alpha_{nj}(z,y,z,v)=\frac{f_j(z,y)-f_j(z,v)}{y-v}.
\]
Hence, from the map germ $\tilde \Delta^1f\colon (\R^{n+1},0)\to\R^{p-n+1}$ defined as
\[
\tilde \Delta^1f(z,y,u)=\left(\frac{f_n(z,u)-f_n(z,y)}{u-y},\ldots, \frac{f_p(z,u)-f_p(z,y)}{u-y}\right),
\]
we may to obtain $\mathcal{L}_0(\tilde{\Delta}f)$ considering the infimum of the numbers $\alpha$ such that
\[
\frac{\|\tilde \Delta^1f(z,y,u)\|}{(|z|+|y|+|u|)^\alpha} \geq C,
\]
for some $\epsilon, C>0$ such that $\|z,y,u\|<\epsilon$.
}
\end{remark}

Let $F\colon(\mathbb{R}\times \mathbb{R}^n,0)\rightarrow (\mathbb{R}\times \mathbb{R}^p,0)$ be an analytic map given by $F(t,x)=(t,f_t(x))$, such that $f_t(0)=0$, for all $t$. We say that $F$ is a 1-parameter family.
We say that a 1-parameter family $F$ is \emph{uniformly injective} (resp. \emph{has uniform isolated singularity}) when there exists a representative $F\colon I \times \mathcal{U}\rightarrow I \times\mathbb{R}^p$ such that each $f_t$, $t \in I$, is injective (resp. is injective and immersion outside the origin) on $\mathcal{U}$.


\begin{example}\label{example}
{\rm Let $f_0\colon (\mathbb{R}^2,0)\rightarrow (\mathbb{R}^4,0)$ given by $f_0(x,y)=(x,y^2,y^3,x^3y)$. This is a map germ of type $II_3$ in the Klotz-Pop-Rieger list of $\mathscr A$-simple map-germs (see \cite{kpr}). Let us consider the family $F(t,(x,y))=(t,f_t(x,y))=(t,(x,y^2,y^3,x^3y+tx^2y))$ (an unfolding of $f_0$). This family is uniformly injective. Moreover, for each $t$, $f_t$ has isolated singularity for $\|(x,y)\|<t$. However, since $(-t,0)$ is a non-immersive point of $f_t$, the family $F$ does not have uniform isolated singularity. In order to obtain $\mathcal{L}_0(\tilde{\Delta}f_t)$, we consider
\[
\tilde{\Delta}^1{f_t}(x,y,u)=(u^2-y^2,u^3-y^3, x^3(y-u)+tx^2(u-y),u+y,u^2+uy+y^2,x^3+tx^2),
\]
hence
\begin{align*}
\frac{\|\tilde{\Delta}^1_{f_t}(x,y,u)\|}{(|x|+|y|+|u|)^{\alpha}}&=\frac{|u^2-y^2|}{(|x|+|y|+|u|)^{\alpha}}+\frac{|u^3-y^3|}{(|x|+|y|+|u|)^{\alpha}}+\frac{|x^3(y-u)+tx^2(u-y)|}{(|x|+|y|+|u|)^{\alpha}}\\&+\frac{|u+y|}{(|x|+|y|+|u|)^{\alpha}}
+\frac{|u^2+uy+y^2|}{(|x|+|y|+|u|)^{\alpha}}+\frac{|x^3+tx^2|}{(|x|+|y|+|u|)^{\alpha}}.
\end{align*}
From the last term in the right hand side of the equality, we conclude that $\mathcal{L}_0(\tilde{\Delta}f_t)=2$ when $t \neq 0$ and $\mathcal{L}_0(\tilde{\Delta}f_0)=3$, when $t=0$.}
\end{example}

\begin{example}[Finitely determined case]

{\rm Let $F(t,x)=(t,f_t(x))$ be a family such that $\delta(f_t)<\infty$ and $\mathcal{L}_0(\tilde{\Delta}f_t)$ is constant. In particular, each $f_t$ is a finitely determined map germ. The following example shows that these conditions are not enough  to ensure that $\delta$ is constant along a family:
\[
 f_t(x,y)=(x,y^2,y(x^2+y^2),y(x^4+y^6+ty^2)).
\]
We have, for each $t$, $\delta(f_t)<\infty$. Thus, for each $t$, $f_t$ is a finitely determined map germ and, hence, it has isolated singularity at $0$. Using the \texttt{Singular} software (see \cite{GP}), we obtain that $\delta(f_t)=2$, for $t \neq 0$ and $\delta(f_0)=4$. On the other hand, $\mathcal{L}_0(\tilde{\Delta}f_t)=2$. Moreover, this family has uniform isolated singularity, because $V(\tilde \Delta f_t)=\{0\}$, for all $t$.}
\end{example}

\begin{example}\label{counterexample2}
{\rm The constancy of the $\delta$-invariant along a family $F(t,x)=(t,f_t(x))$ implies that it has uniformily isolated singularity (see \cite[Lemma 4.2]{MB}). From this (and considering the previous example), it seems natural to ask if the constancy of $
\mathcal{L}_0(\tilde{\Delta}f_t)$ implies uniform isolated singularity. In fact, the question has negative answer as the following example shows:
\[
f_t(x,y)=(x,y^2,y(x^4+txy^2),y(y^4+txy^2)).
\]
We will prove that $\mathcal{L}_0(\tilde{\Delta}f_t)=4$, for all $t$.
 If $t=0$, then
 \[\frac{\|\tilde{\Delta}{f_0}(x,y,u)\|}{(|x|+|y|+|u|)^{\alpha}}=\left(0, \frac{|y+u|}{(|x|+|y|+|u|)^{\alpha}},\frac{x^4}{(|x|+|y|+|u|)^{\alpha}},\frac{|u^4+u^3y+u^2y^2uy^3+y^4|}{(|x|+|y|+|u|)^{\alpha}}\right),
 \]
so
\[
\frac{\|\tilde{\Delta}_{f_t}(x,y,-y)\|}{(|x|+2|y|)^{\alpha}}=(0,0,\frac{x^4}{(|x|+2|y|)^{\alpha}}, \frac{y^4}{(|x|+2|y|)^{\alpha}}),
\]
which implies that $\mathcal{L}_0(\tilde{\Delta}f_0)=4$.

Otherwise, if $t\neq 0$, we have
\[
\frac{\|\tilde{\Delta}{f_t}(x,y,u)\|}{(|x|+|y|+|u|)^{\alpha}}=\left(0, \frac{|y+u|}{(|x|+|y|+|u|)^{\alpha}},\frac{x^4+tx(u^2+uy+y^2)}{(|x|+|y|+|u|)^{\alpha}},P\right),
\]
with
\[
P=\frac{|u^4+u^3y+u^2y^2uy^3+y^4+tx(u^2+uy+y^2)|}{(|x|+|y|+|u|)^{\alpha}}.
\]
In particular,
\[
\frac{\|\tilde{\Delta}{f_t}(x,y,-y)\|}{(|x|+2|y|)^{\alpha}}=\left(0,0,\frac{|x^4+txy^2|}{(|x|+2|y|)^{\alpha}}, \frac{|y^4+txy^2|}{(|x|+2|y|)^{\alpha}}\right).
\]
For $y=0$, this gives $\frac{\|\tilde{\Delta}{f_t}(x,0,0)\|}{(|x|)^{\alpha}}=\left(0,0,\frac{x^4}{(|x|)^{\alpha}},0\right)$, hence we must take
$\alpha=4$.
For $y\neq 0$, the last component vanishes along the curve $x= -\frac{y^2}{t}$. It follows that
\[
\frac{\|\tilde{\Delta}{f_t}(-\frac{y^2}{t},y,-y)\|}{(\frac{y^2}{|t|}+2|y|)^{\alpha}}=
\left(0,0,\frac{y^4(\frac{y^4-t^4}{t^4})}{(\frac{y^2}{|t|}+2|y|)^{\alpha}},0\right).
\]
Hence, for each representative $f_t$ such that $|y|<t$, we must take $\alpha=4$. Thus, the claim is proved.

However, For $|y|=t$, we have that $\|\tilde{\Delta}{f_t}(-\frac{y^2}{t},y,-y)\|=0$. It means that each choice of a representative $f_t$ along of this family depends of $t$, i.e., this family does not have uniform isolated singularity.}

\end{example}

\section{$C^0$-Finite determinacy}
As we have seen, an analytic map germ $f\colon (\mathbb{R}^n,0)\to  (\mathbb{R}^p,0) \ (n\leq p)$ with isolated singularity does not have, in general, the finite determinacy property. Inspired on the question in Wall's paper (see \cite{Wall}), in this section we consider the $C^0$-finite determinacy. 
For analytic function germs $h\colon(\R^n,0)\to\R$, it is well known that isolated singularity implies $C^0$-finite determinacy (see Kuo \cite{Kuo}) and there are many papers related to find estimates for the degree of $C^0$-finite determinacy (see for instance \cite{BA1}). In this section, we show that, for real analytic maps germs $f$ from $\mathbb{R}^2$ to $\mathbb{R}^4$, isolated singularity implies $C^0$-finite determinacy (Theorem \ref{finitedeterminacy}). The steps to prove this are the following: For maps $f$ from $\mathbb{R}^n$ to $\mathbb{R}^p$, $n\leq p$, we establish a condition for a given unfolding of $f$ to have uniform isolated singularity and a Milnor-Fukuda radius of $f_t$ that does not depend of $t$ (Proposition \ref{uniformzerop} and Proposition \ref{regularvalue}). In particular, we obtain that this unfolding, for maps from $\mathbb{R}^2$ to $\mathbb{R}^4$, has constant knot type.

\medskip
Consider $\mathscr E_{n+1}$ the local ring of analytic function germs $(\R\times \R^n,0)\to\R$ and $\mathscr E_{n+1}^p$ the $\mathscr E_{n+1}$-module of analytic map germs $(\R\times\R^n,0)\to\R^p$. We also denote by $\mathcal{M}_n$ the  ideal of $\mathscr E_{n+1}$ generated by $x_1,\dots,x_n$. The following lemma is probably well known, but we include here a proof for completeness. It says that if $g_t$ is a 1-parameter family of map germs, such that $g_0$ has isolated zero at the origin, then the Lojasiewicz exponent of $g_0$ controls the fact that the family has uniform isolated zero at the origin.

\begin{lemma}\label{uniformzero}
Let $G\colon(\mathbb{R}\times \mathbb{R}^n,0)\rightarrow (\mathbb{R}\times \mathbb{R}^p,0)$, $n\leq p$, be a $1$-parameter family given by $G(t,x)=(t,g_t(x))$ and such that $g_0^{-1}(0)=\{0\}$.
Suppose that $g_t-g_0 \in \mathcal{M}_n^{[\alpha_0]+1}\mathscr E_{n+1}^p$, where $\alpha_0$ is a Lojasiewicz exponent of $g_0$ and  $[\alpha_0]$ is the nearest integer less than $\alpha_0$. Then, there exists a representative $G\colon(-\delta,\delta)\times \mathcal{U} \rightarrow (\delta,\delta)\times \mathbb{R}^p$ such that  $g_t^{-1}(0)=\{0\}$, for all $t \in (-\delta,\delta)$.
\end{lemma}

\begin{proof}
Let $m=[\alpha_0]+1$ and write  $g_t(x)=g_0(x)+th_t(x)$, so that $h_t\in \mathcal{M}_n^{m}\mathscr E_{n+1}^p$. We first claim that there exist a representative $G\colon (-\delta,\delta)\times \mathcal{U}\rightarrow (-\delta,\delta)\times \mathbb{R}^p$ and $B>0$ such that $\|h_t(x)\|\leq B\|x\|^{\alpha_0}$, $t \in (-\delta,\delta)$, $x\in\mathcal U$.

In fact, $\mathcal{M}_n^{m}$ is generated by the monomials $x^\beta=x_1^{\beta_1}\dots x_n^{\beta_n}$, with $\beta=(\beta_1,\dots,\beta_n)$ such that $|\beta|=\beta_1+\dots+\beta_n=m$. Hence, we can write
\[
h_t(x)=\sum_{|\beta|=m}x^\beta H_\beta(x,t),
\]
for some $H_\beta\in\mathscr E_{n+1}^p$. We fix representatives in $(-\delta',\delta')\times \mathcal{U}'$ and take $0<\delta<\delta'$ and $\mathcal U=B_\epsilon(0)$ such that $\overline B_\epsilon(0)\subset \mathcal U'$. If $x=0$, the inequality is immediate. Let $x\ne0$ and put $r=\|x\|$ and $u=x/\|x\|$. We have
\[
\|h_t(x)\|\le r^m\sum_{|\beta|=m}\|u^\beta H_\beta(ru,t)\|\le B r^{\alpha_0},
\]
where $B=\sup\{\sum_{|\beta|=m}\|u^\beta H_\beta(ru,t)\|:\ u\in \mathbb S^{n-1},\ r\in[0,\epsilon],\ t\in[-\delta,\delta]\}$. This concludes the proof of the claim.

Now we can finish the proof of the lemma. Fix a representative $G\colon (-\delta,\delta)\times \mathcal{U}\rightarrow (-\delta,\delta)\times \mathbb{R}^p$
such that $\|g_0(x)\| \geq C\|x\|^{\alpha_0}$ and $\|h_t(x)\|\leq B\|x\|^{\alpha_0}$, for all $(t,x) \in (-\delta,\delta)\times\mathcal U$. Moreover, we can assume that $0<\delta<\frac{C}{2B}$. For all $(t,x) \in (-\delta,\delta)\times\mathcal U$,
\[
\|g_t(x)\|=\|g_0(x)+th_t(x)\| \geq \|g_0(x)\|-|t|\|h_t(x)\| \geq (C-|t|B)\|x\|^{\alpha_0}>\frac{C}{2}\|x\|^{\alpha_0}.
\]
Hence,  $g_t^{-1}(0)=\{0\}$, for all $t \in (-\delta,\delta)$.	
\end{proof}

We use Lemma \ref{uniformzero} to show that if $f_t$ is a 1-parameter family of map germs such that $f_0$ has isolated singularity, then the isolated singularity exponent $\mathcal{L}_0(\tilde{\Delta}f_0)$ controls that $f_t$ has uniform isolated singularity.

\begin{proposition}[Uniform isolated singularity]\label{uniformzerop}
Let $F\colon(\mathbb{R}\times \mathbb{R}^n,0)\rightarrow (\mathbb{R}\times \mathbb{R}^p,0)$ be a family $F(t,x)=(t,f_t(x))$ such that $f_0$ has isolated singularity. Assume that 
$f_t-f_0 \in \mathcal{M}_n^k\mathscr{E}_{n+1}^p$, where $k\geq [\mathcal{L}_0(\tilde{\Delta}f_0)]+2$. Then, the family has uniform isolated singularity.
\end{proposition}

\begin{proof}
We consider the map germ $\tilde \Delta F\colon(\R\times\mathbb{R}^{2n},0)\rightarrow (\mathbb{R}^{p+r},0)$ given by
\[\tilde 
\Delta F(t,x,x')=\tilde \Delta f_t(x,x')=(f_t(x')-f_t(x),D_{1t}(x,x'),\ldots,D_{rt}(x,x')).
\] 
Let $\alpha_0=\mathcal{L}_0(\tilde{\Delta}f_0)$. We need to show that  $\tilde \Delta f_t-\tilde \Delta f_0 \in \mathcal{M}_{2n}^{[\alpha_0]+1}\mathscr{E}_{2n+1}^{p+r}$. By Lemma \ref{uniformzero}, this implies that there exists a representative $\Delta F\colon (-\delta,\delta)\times \tilde{\mathcal{U}} \rightarrow (-\delta,\delta)\times \mathbb{R}^{p+r}$ such that $V(\tilde \Delta f_t)=\Delta f_t^{-1}(0)=\{0\}$. Thus, $f_t$ is injective and an immersion outside the origin, for all $t \in (-\delta,\delta)$.

Since $k> [\alpha_0]+1$, it is clear that $f_t(x')-f_t(x)-(f_0(x')-f_0(x)) \in\mathcal{M}_{2n}^{[\alpha_0]+1}\mathscr{E}_{2n+1}^p$. So, we only need to show that  ${D}_{it}-D_{i0}\in \mathcal{M}_{2n}^{[\alpha_0]+1}$, for all $i=1,2,\ldots,r$. By construction, $D_{it}$ are the $n$-minors of a $p\times n$ matrix $\alpha_t$ with entries in $\mathscr{E}_{2n+1}$ such that
\[
f_{t}(x')-f_{t}(x)=\alpha_{t}(x,x')(x'-x).
\]
Thus,
\[
f_t(x')-f_t(x)-(f_0(x')-f_0(x))=(\alpha_{t}(x,x')-\alpha_0(x',x))(x'-x)=t\beta_t(x',x)(x'-x),
\]
for some matrix $\beta_t$. Since the components of $f_t(x')-f_t(x)-(f_0(x')-f_0(x))$ are in $ \mathcal{M}_{2n}^k$, it follows that the entries of $\beta_t$ must be in $ \mathcal{M}_{2n}^{k-1}$. Take one of the minors $D_{it}$. Without loss of generality, we can assume that it is given by the first $n$ rows of $\alpha_t$. That is, $D_{it}=[v_{1t},\dots,v_{nt}]$, where $v_{jt}$ are the row vectors of $\alpha_t$. Then, $v_{jt}=v_{j0}+tw_{jt}$, where $w_{jt}$ are the row vectors of $\beta_t$. We have:
\begin{align*}
D_{it}&=[v_{1t},\dots,v_{nt}]=[v_{10}+tw_{1t},\dots,v_{n0}+tw_{nt}]\\
&=[v_{10},\dots,v_{n0}]+t\sum_{j=1}^n [v_{10},\dots,w_{jt},\dots,v_{n0}]+\dots+t^n [w_{1t},\dots,w_{nt}]\\
&=D_{i0}+t\sum_{j=1}^n [v_{10},\dots,w_{jt},\dots,v_{n0}]+\dots+t^n [w_{1t},\dots,w_{nt}].
\end{align*}
Therefore, $D_{it}-D_{i0} \in \mathcal{M}_{2n}^{k-1}\subset  \mathcal{M}_{2n}^{[\alpha_0]+1}$, since $k-1\ge [\alpha_0]+1$.
\end{proof}

The next step is to control that we can take a constant Milnor-Fukuda radius in the family  $f_t$. We recall that $\epsilon_0>0$ is a Milnor-Fukuda radius for a map germ $f$ if $f$ is transverse to all the spheres $S_\epsilon^{p-1}$, with  $0<\epsilon\le\epsilon_0$. This is equivalent to the fact that $\epsilon$ is a regular value of the function $\|f\|^2$,  for all $0<\epsilon\le\epsilon_0$. Again this fact will be controlled by the Lojasiewicz exponent of $\Grad \|f\|^2$, the gradient of $\|f\|^2$.

\begin{proposition}[Uniform regular value]\label{regularvalue}
Let $F\colon(\mathbb{R}\times \mathbb{R}^n,0)\rightarrow (\mathbb{R}\times \mathbb{R}^p,0)$ be a family $F(t,x)=(t,f_t(x))$ such that $f_0$ has isolated singularity. Let $\beta_0=\mathcal{L}_0(\Grad(\|f_0\|^2)$ be the Lojasiewicz exponent of $\Grad(\|f_0\|^2)$. Suppose that $f_t-f_0 \in \mathcal{M}_n^{k}\mathscr{E}_{n+1}^p$, with $k\geq [\beta_0]+1$. Then, there exists $\epsilon_0>0$ and a representative $F\colon(-\delta,\delta)\times\mathcal U\to (-\delta,\delta)\times\R^p$ such that $\epsilon$ is a regular value of $\|f_t\|^2$, for all $0<\epsilon\le\epsilon_0$ and for all $t\in(-\delta,\delta)$. 
\end{proposition}
\begin{proof}
Let $g_t=\|f_t\|^2$. We first show that $0$ is an isolated critical point of $g_0$. Suppose this is not true.  Then, by the curve selection lemma, there exists a non constant analytic arc $\alpha\colon(-\delta,\delta)\rightarrow \mathbb{R}^n$, with $\alpha(0)=0$, such that $\frac{\partial g}{\partial x_i}(\alpha(s))=0$, for all $i=1\ldots,n$. Thus,
\[
 (g\circ\alpha)'(s)=\frac{\partial g}{\partial x_1}(\alpha(s))\alpha_1'(s)+\ldots+\frac{\partial g}{\partial x_n}(\alpha(s))\alpha_n'(s)=0,
\]
for all $s\in(-\delta,\delta)$. Hence $g\circ\alpha=0$ and $f\circ\alpha=0$, but this gives a contradiction, because $f^{-1}(0)=\{0\}$.

Now, we use Lemma \ref{uniformzero} for the family $H(t,x)=(t,h_t(x))$, with $h_t=\Grad g_t$.
We show that $h_t-h_0 \in \mathcal{M}_n^{[\beta_0]+1}\mathscr{E}_{n+1}^p$. 
Write $f_t=f_0+t\beta_t$ for some $\beta_t\in \mathcal{M}_n^{k}\mathscr{E}_{n+1}^p$.
Take a component of $h_{it}$ of $h_t$. We have:
\begin{align*}
h_{it}&=\frac{\partial g_t}{\partial x_i}=2\sum_{j=1}^n f_{jt}\frac{\partial f_{jt}}{\partial x_i}\\
&=2\sum_{j=1}^n (f_{j0}+t \beta_{jt})\left(\frac{\partial f_{j0}}{\partial x_i}+t\frac{\partial \beta_{jt}}{\partial x_i}\right)\\
&=h_{i0}+2t\left(\sum_{j=1}^n f_{j0}\frac{\partial \beta_{jt}}{\partial x_i}+\sum_{j=1}^n \beta_{jt}\frac{\partial f_{j0}}{\partial x_i}\right)+2t^2\sum_{j=1}^n \beta_{jt}\frac{\partial \beta_{jt}}{\partial x_i}.
\end{align*}
Thus, $h_t-h_0 \in \mathcal{M}_n^{k}\mathscr{E}_{n+1}^p\subset \mathcal{M}_n^{[\beta_0]+1}\mathscr{E}_{n+1}^p$, since $k \ge [\beta_0]+1$. By Lemma \ref{uniformzero}, there exists a representative $H\colon(-\delta,\delta)\times \mathcal{U} \rightarrow(-\delta,\delta)\times \mathbb{R}^n$ such that $h_t^{-1}(0)=\{0\}$, for all $t \in (-\delta,\delta)$. In other words, $g_t$ has only a critical point at the origin, for all $t \in (-\delta,\delta)$. In particular, $0$ is the only critical value of $g_t$, for all $t \in (-\delta,\delta)$.
\end{proof}

Finally, we arrive to the main result of this section, which proves that any analytic map germ $f\colon(\mathbb{R}^{{2}},0) \rightarrow (\mathbb{R}^{{4}},0)$ with isolated singularity is $C^0$-finitely determined. In fact, the theorem gives an estimate of the degree of $C^0$-finite determinacy in terms of the Lojasiewicz exponents.

\begin{theorem}\label{finitedeterminacy}
Let $f\colon(\mathbb{R}^{{2}},0) \rightarrow (\mathbb{R}^{{4}},0)$ be an analytic map germ with isolated singularity. Let $k \geq \max\{[\mathcal{L}_0(\tilde{\Delta}f)]+2, [\mathcal{L}_0(\Grad(\|f\|^2))]+1\}$. Then, $f$ is $C^0$-$k$-determined.
\end{theorem}
\begin{proof}
Given an analytic map $g\colon(\mathbb{R}^{{2}},0) \rightarrow (\mathbb{R}^{{4}},0)$ such that $j^kg(0)=j^kf(0)$, we write $g=j^kf(0)+h$, where  $h \in \mathcal{M}_2^{k+1}\mathscr{E}_2^4$ and consider the family $G(t,x)=(t,g_t(x))$ given by $g_t(x)=j^kf(0)(x)+th(x)$. Fix a parameter $t_0\in\R$, by Proposition 
\ref{uniformzerop} and Proposition \ref{regularvalue}, there exists a representative $G\colon(t_0-\delta,t_0+\delta)\times \mathcal{U}\rightarrow (t_0-\delta,t_0+\delta)\times \R^4$ such that $g_t$ is injective and an immersion outside the origin and there exists $\epsilon_0>0$ such that for all $0<\epsilon \leq \epsilon_0$, $\epsilon$ is a regular value of each $g_t$, for all $t \in (t_0-\delta,t_0+\delta)$. This implies that $\epsilon_0$ is a Milnor-Fukuda radius for each map $g_t$, with $t \in (t_0-\delta,t_0+\delta)$ (see Theorem \ref{conestruct}). 

Now, the map $G$ defines an isotopy between the family of knots 
$K(g_t)=g_t(\mathcal{U})\cap {S}^3_{\epsilon_0}$ (see \cite[Theorem 4.4]{MB}). Hence, by Theorem \ref{equivalence}, all the map germs $g_t$ with $t \in (t_0-\delta,t_0+\delta)$ are $C^0$-$\mathscr A$-equivalent. Since $\R$ is connected, this implies that all the map germs $g_t$ with $t \in\R$ are $C^0$-$\mathscr A$-equivalent. In particular, $g$ and $f$ are $C^0$-$\mathscr A$-equivalent.
\end{proof}

\section{The double point exponent}
  Let $f\colon(\mathbb{R}^n,0) \rightarrow (\mathbb{R}^p,0), \ (n\leq p)$ be an analytic map germ. Let us consider the map germ  $\Delta f\colon(\mathbb{R}^{2n},0)\rightarrow \mathbb{R}^p, \ \Delta f(x,x')=f(x)-f(x') $. Let $\alpha_0$ be the \emph{Lojasiewicz exponent} of the map $\Delta f$. Then, there exists positive constants $C, \epsilon, \alpha$ such that the following \emph{Lojasiewicz inequality} holds:
\begin{equation}\label{loja}
\|\Delta f(x,x')\| \geq C\dist((x,x'), V(\Delta f))^{\alpha},
\end{equation}
for all $\|(x,x')\| < \epsilon$, and $\alpha_0$ is the infimum of the exponents $\alpha$ such that the inequality (\ref{loja}) is true. Notice that if $f$ is a injective map, we have that $V(\Delta f))=\{(x,x); x \in \mathcal{U}\}$, for some representative. In this case, the inequality (\ref{loja}) looks as follows:
\begin{equation}\label{loja2}
 \|f(x)-f(x')\| \geq C\|x-x'\|^{\alpha},
\end{equation}
for all $\|(x,x')\| < \epsilon$.

\begin{definition}
\emph{Given an analytic map germ $f\colon(\mathbb{R}^n,0) \rightarrow (\mathbb{R}^p,0), \ (n\leq p)$, the infimum $\alpha_0$ of exponents $\alpha$ such that the inequality (\ref{loja2}) holds is called the \emph{double point exponent of $f$}. We write $\alpha_0=\mathcal{L}_0(\Delta f)$. }.
\end{definition}
\begin{remark}{\rm 
$\mathcal{L}_0(\Delta f)<\infty$ if and only if $f$ is  injective.}
\end{remark}

\begin{example}
{\rm Let $\gamma\colon(\mathbb{R},0) \rightarrow (\mathbb{R}^3,0)$ be a germ of an analytic arc with isolated singularity, where $\gamma$ is injective. Let $f\colon(\mathbb{R}^2,0)\rightarrow (\mathbb{R}^4,0)$ be the constant unfolding $f(x,y)=(x,\gamma(y))$. We have that $\tilde{\Delta} f(x,y,u)=(0,\gamma(u)-\gamma(y), \Delta_1\gamma(y,u))$, where $\Delta_1\gamma(y,u)=(\gamma(u)-\gamma(y))/(u-y)$. If $\tilde{\Delta}_f(x,y,u)=0$ then $u=y$ and hence, $0=\Delta_1\gamma(y,y)=\gamma'(y)$. Therefore, $y=u=0$. It means that $V(\tilde{\Delta} f)=\{(x,0,0)\}$, i.e., the non-immersive locus of $f$ is the $x$-axes. By Proposition \ref{is}, $\mathcal{L}_0(\tilde{\Delta}_f)=\infty$.

On the other hand, if $\gamma(y)$ contains at least a monomial of odd  degree $y^{2m+1}$, we have $\mathcal{L}_0({\Delta}_f)=2l$, where $2l+1$ is the smallest odd exponent of a monomial in $\gamma(y)$.}
\end{example}

\begin{lemma}
Let $f\colon(\mathbb{R}^n,0) \rightarrow (\mathbb{R}^p,0), \ (n\leq p)$ be analytic and injective.  Then $\alpha_0=\mathcal{L}_0(\Delta f) \geq 1$. In particular $f^{-1}$ is a $\frac{1}{\alpha_0}$ continuous H\"older map on the image.
\end{lemma}
\begin{proof}
Since $f$ is analytic, it is locally Lipschitz near $0$. Thus, there exist constants $\tilde C,\delta_1>0$ such that
\[
\frac{\|f(x)-f(x')\|}{\|x-x'\|} \leq \tilde C,
\]
for all  $0<\|(x,x')\|< \delta_1$. On the other hand, there exist positive constants $C,\delta_2,\alpha_0$ such that 
\begin{equation}\label{eq:alpha}
\frac{\|f(x)-f(x')\|}{\|x-x'\|^{\alpha_0}}>C,
\end{equation}
for all $\|(x,x')\|< \delta_2$.  Let $\epsilon=\min\{\delta_1,\delta_2\}$. If $\alpha_0<1$, we can write
\[
\frac{\|f(x)-f(x')\|}{\|x-x'\|^{\alpha_0}}= \frac{\|f(x)-f(x')\|}{\|x-x'\|}\|x-x'\|^{1-\alpha_0}\to 0,
\]
when $\|(x,x')\| \to 0$. But this a contradiction with (\ref{eq:alpha}) and thus, we must have $\alpha_0 \geq 1$. 
\end{proof}

Let $f,g\colon (\mathbb{R}^n,0){ \rightarrow} (\mathbb{R}^p,0)$ be analytic map germs. Then $f$ and $g$ are called bi-Lipschitz equivalent if there exist bi-Lipschitz homeomorphisms $\phi\colon (\mathbb{R}^n,0) {\rightarrow }(\mathbb{R}^n,0)$ and $\psi\colon (\mathbb{R}^p,0) \to  (\mathbb{R}^p,0)$ such that $g=\psi\circ f\circ \phi^{-1}$. In the next lemma, we show that $\mathcal{L}_0(\Delta f)$ is a bi-Lipschitz invariant.

\begin{lemma}\label{analytic invariant}
\emph{If $f$ and $g$ are bi-Lipschitz equivalent then, $\mathcal{L}_0(\Delta f)=\mathcal{L}_0(\Delta g)$}.
\end{lemma}

\begin{proof}
First, it is immediate that $\mathcal{L}_0(\Delta f) = \infty$ if and only if $\mathcal{L}_0(\Delta g) = \infty$, so we can assume both numbers are finite. 
Let $\alpha_0 =\mathcal{L}_0(\Delta f)<\infty$ and assume that $g = f \circ \phi$, where $\phi\colon(\mathbb{R}^n,0)\rightarrow (\mathbb{R}^n,0)$ is a bi-Lipschitz homeomorphism.  We have:
\[
\|g(x)-g(x')\|=\|f \circ \phi(x)- f \circ \phi(x')\|\geq C\|\phi(x)-\phi(x')\|^{\alpha_0} \geq C\tilde C\|x-x'\|^{\alpha_0},
\]
for some positive constants $C, \tilde C$ and for all $\|(x,x')\|<\epsilon$. Thus, $\mathcal{L}_0(\Delta g) \leq \alpha_0=\mathcal{L}_0(\Delta f)$. Using the same argument for $f=g \circ \phi^{-1}$, we obtain   $\mathcal{L}_0(\Delta g) \geq\mathcal{L}_0(\Delta f)$. Hence,  $\mathcal{L}_0(\Delta g)=\mathcal{L}_0(\Delta f)$. 

On the other hand, assume that $g=\psi \circ f$, where $\psi\colon(\mathbb{R}^p,0)\rightarrow (\mathbb{R}^p,0)$ is a bi-Lipschitz homeomorphism. We have:
\[
\|g(x)-g(x')\|=\|\psi(f(x))-\psi(f(x'))\|\geq D\|f(x)-f(x')\| \geq D \tilde D \|x-x'\|^{\alpha_0},
\]
 for some positive constants $D, \tilde D$ and for all $\|(x,x')\|<\epsilon$. Thus, $\mathcal{L}_0(\Delta g) \leq \mathcal{L}_0(\Delta f)$ and, in a similar way (by considering $\psi^{-1} \circ g$), we obtain  $\mathcal{L}_0(\Delta g) \geq \mathcal{L}_0(\Delta f)$, so $\mathcal{L}_0(\Delta g)=\mathcal{L}_0(\Delta f)$. 
 
 Now, if $ g=\psi \circ f \circ \phi^{-1}$, then $\mathcal{L}_0(\Delta g)=\mathcal{L}_0(\Delta {f \circ \phi})=\mathcal{L}_0(\Delta f)$.
\end{proof}


The main result of this section is that any Lipschitz embedding is a smooth embedding.

\begin{theorem}\label{embedding}
Let $f\colon(\mathbb{R}^n,0) \rightarrow (\mathbb{R}^p,0)$, $n \leq p$, be analytic and injective. Then, the following assertions are equivalent:
\begin{enumerate}
\item $f$ is a bi-Lipschitz map on the image;
\item $\mathcal{L}_0(\Delta f)=1$;
\item $f$ is a smooth embedding.
\end{enumerate}
\end{theorem}

\begin{proof}
Notice that $(2) \Leftrightarrow (1)$ is immediate from the definition of $\mathcal{L}_0(\Delta f)$ and that $(3) \Rightarrow (1)$ is also clear. So, we only have to prove $(2) \Rightarrow (3)$. Assume that $0$ is a non-immersive point of $f$. Let $0<r\leq n$ be the corank of $f$ at $0$. After $\mathscr A$-equivalence, we may write
\begin{equation}\label{cc}
f(z,y)=(z,\tilde{f}(z,y)), \ z \in \mathbb{R}^{n-r}, \ y \in \mathbb{R}^r,
\end{equation}
where $\tilde{f}(z,y)=(f_{n-r+1}(z,y),\ldots f_p(z,y))$ and $f_{n+r+1},\dots,f_p\in\mathcal M_n^2$.

For each $f_i$, with $i=n-r+1, \ldots, p$, we write
\[
f_i(z,u)-f_i(z,y)=\sum_{j=1}^r \alpha_{ij}(z,y,u)(y_j-u_j),
\]
for some $\alpha_{ij} \in \mathcal{M}_{r+n}$. By using matrix notation $\alpha=(\alpha_{ij})$, we have
\[
\tilde f(z,u)-\tilde f(z,y)=\alpha(z,y,u)(u-y),
\]
with $\alpha(0)=0$. Let $\|\alpha\|$ be the matrix norm of $\alpha$. If $u\ne y$, then:
\begin{equation}\label{limit}
\frac{\|\tilde f(z,u)-\tilde f(z,y)\|}{\|u-y\|}=\frac{\|\alpha(z,y,u)(u-y)\|}{\|u-y\|}\le \frac{\|\alpha(z,y,u)\|\|(u-y)\|}{\|u-y\|}=\|\alpha(z,y,u)\|\to 0,
\end{equation}
when $\|(z,y,u)\|\to 0$.

On the other hand, assume that $\mathcal{L}_0(\Delta f)=1$. If $(z,y)\ne(z',u)$, then
\[
\frac{\|z'-z\|+\|\tilde f(z',u)-\tilde f(z,y)\|}{\|z'-z\|+\|u-y\|}\ge C,
\]
for some $C>0$ and for all $\|(z,y,z',u)\|<\epsilon$. In particular, for $z=z'$ and $y\ne u$,
\[
\frac{\|\tilde f(z,u)-\tilde f(z,y)\|}{\|u-y\|}\ge C,
\]
for all $\|(z,y,u)\|<\epsilon$. But this gives a contradiction with (\ref{limit}).
\end{proof}

\begin{corollary} Let $X\subset\R^p$ be an analytic subset. If $X$ locally parametrized at $x_0$ as the image of an analytic map which is bi-Lipschitz onto its image, then $X$ is smooth at $x_0$.
\end{corollary}

In the complex case, Birbrair, L\^e, Fernandes and Sampaio show in \cite{BFLS} that if a complex algebraic set $X\subset\C^n$ is Lipschitz regular at a point $x_0\in X$, then $X$ is smooth at $x_0$. This is false in the real case, for instance, the surface in $\R^3$ given by $x^3+y^3=z^3$ is Lipschitz regular at the origin, but it is not smooth (see \cite{tese-Sampaio}). Thus, our corollary can be seen as a weaker real version of that theorem.

\begin{remark}{\rm The same argument of the proof of Theorem \ref{embedding} works is $f$ not analytic, but is $C^1$. In that case, the exponent $\mathcal L_0(\Delta f)$ is not defined, but one can show that if $f$ is a $C^1$ map which is bi-Lipschitz onto its image, then $f$ is a $C^1$-embedding.}
\end{remark}

\bigskip

\noindent{\bf Acknowledgements}. We would like to thank Osamu Saeki and Carles Bivi\`a-Ausina by the extremely useful suggestions.

\end{document}